\documentclass{amsart}

\usepackage{amsmath,amssymb,amsthm,mathtools}
\usepackage{microtype}
\usepackage[bookmarks=false,hidelinks]{hyperref}

\numberwithin{equation}{section}

\newtheorem{theorem}{Theorem}[section]
\newtheorem{proposition}[theorem]{Proposition}
\newtheorem{lemma}[theorem]{Lemma}
\newtheorem{corollary}[theorem]{Corollary}
\theoremstyle{definition}
\newtheorem{definition}[theorem]{Definition}
\newtheorem{example}[theorem]{Example}
\theoremstyle{remark}
\newtheorem{remark}[theorem]{Remark}

\newcommand{\Z}{\mathbb Z}
\newcommand{\C}{\mathbb C}
\newcommand{\supp}{\operatorname{supp}}
\newcommand{\1}{\mathbf 1}

\title[Orlicz spaces and the uncertainty principle]{Orlicz spaces and the uncertainty principle}

\author{Alex Iosevich}
\address{Department of Mathematics, University of Rochester, Rochester, NY 14627, USA}
\email{iosevich@math.rochester.edu}

\author{Isaac Li}
\address{Department of Mathematics, University of Rochester, Rochester, NY 14627, USA}
\email{ili3@u.rochester.edu}

\author{Zhihe Li}
\address{Department of Mathematics, University of Rochester, Rochester, NY 14627, USA}
\email{zli154@ur.rochester.edu}

\author{Ke Yu}
\address{Department of Mathematics, University of Rochester, Rochester, NY 14627, USA}
\email{kyu26@ur.rochester.edu}

\thanks{Alex Iosevich was supported in part by the National Science Foundation under grant DMS-2154232.}

\begin{document}

\begin{abstract}
Let $f:\Z_N^d\to\C$ be a nonzero signal.  The classical finite uncertainty principle asserts that
\[
 |\supp(f)|\,|\supp(\widehat f)|\geq N^d.
\]
Iosevich and Mayeli obtained stronger conclusions when the Fourier support satisfies a uniform Fourier restriction estimate.  We extend this framework from Lebesgue spaces to Orlicz spaces.  We introduce $(\Phi,\Psi)$ restriction estimates and derive support uncertainty principles, strong annihilating-pair inequalities, and exact-recovery criteria.  We also study Fourier supports that are uniform $\Lambda_\Phi$ sets.  A support-transfer argument shows that if
\[
 \|u\|_{L^\Phi(\mu)}\leq K\|u\|_{L^2(\mu)}
\]
holds uniformly for all functions whose Fourier transform is supported in a fixed set, then every such nonzero function with spatial support of relative size $\delta$ satisfies
\[
 \delta\bigl[\Phi^{-1}(\delta^{-1})\bigr]^2\geq K^{-2}.
\]
Combining this with probabilistic $\Lambda_\Phi$ results of Ryou and Burstein yields uncertainty and sparse-recovery statements for random Fourier sets.  In particular, for sufficiently large $N$, Burstein's endpoint construction gives, with probability at least $1/4$, Fourier sets of cardinality at least $N^d/(e\log^{\alpha+1}(N^d))$ for which every nonzero Fourier-supported signal satisfies
\[
 |\supp(f)|\geq \frac{N^d}{[\log(N^d)]^{C_\alpha}}.
\]
Thus the spatial support loses at most a power of the logarithm of the ambient cardinality.
\end{abstract}

\maketitle

\tableofcontents

\section{Introduction}

Write $M=N^d$.  For a signal $f:\Z_N^d\to\C$, with Fourier transform normalized as in Section~\ref{sec:preliminaries}, the classical uncertainty principle of Matolcsi--Sz\H{u}cs and Donoho--Stark~\cite{MatolcsiSzucs,DonohoStark} states that
\begin{equation}\label{eq:classical-up-intro}
 |\supp(f)|\,|\supp(\widehat f)|\geq M.
\end{equation}
This principle is closely connected with exact recovery.  If a signal is known to be supported on a set of cardinality at most $s$ and its Fourier coefficients are unobserved on a set $S$, then uniqueness follows whenever no nonzero signal with Fourier support in $S$ can be supported on at most $2s$ points.

Iosevich and Mayeli~\cite{IosevichMayeli} strengthened \eqref{eq:classical-up-intro} under a uniform $(p,q)$ Fourier restriction hypothesis on the Fourier support.  Quantitative strong annihilating-pair versions were subsequently obtained by Jaming, Iosevich, and Mayeli~\cite{JamingIosevichMayeli}.  The purpose of this paper is to develop the corresponding theory in Orlicz spaces.

Let $\Phi$ and $\Psi$ be Young functions.  We say that a set $S\subset\Z_N^d$ satisfies a $(\Phi,\Psi)$ restriction estimate if
\begin{equation}\label{eq:intro-restriction}
 \|\widehat g\|_{L^\Psi(\mu_S)}
 \leq C_{\Phi,\Psi}(S)M^{-1}\|g\|_{L^\Phi(\Z_N^d)}
\end{equation}
for every $g:\Z_N^d\to\C$, with a constant independent of $N$.  Our first group of results converts \eqref{eq:intro-restriction} into support inequalities.  Two complementary arguments are used.  The first combines Fourier inversion with a Fourier-support estimate for the Orlicz norm.  The second uses $L^2$ comparison, normalized-versus-counting norm conversion, and, in the subquadratic range, a finite-group Orlicz Hausdorff--Young estimate.  When $\Phi(t)=t^p$ and $\Psi(t)=t^q$, the second argument recovers the estimates of~\cite{IosevichMayeli}.

We next prove strong annihilating-pair inequalities.  Rather than repeating the full earlier $L^p$ statements, we recall only their role: a support inequality is upgraded to a stable estimate controlling the entire signal by its spatial mass outside $E$ and its Fourier mass outside $S$.  The Orlicz result appears in Theorem~\ref{thm:strong-annihilating-pair}, and the corresponding support consequence is Corollary~\ref{cor:annihilating-support}.

The second part of the paper concerns $\Lambda_\Phi$ sets.  A fixed set $S$ is a uniform $\Lambda_\Phi$ set with constant $K$ if
\begin{equation}\label{eq:intro-lambda}
 \|u\|_{L^\Phi(\mu)}\leq K\|u\|_{L^2(\mu)}
\end{equation}
for every $u$ satisfying $\supp(\widehat u)\subset S$.  Uniformity over the entire Fourier-supported subspace is essential: for a single fixed function, a constant in \eqref{eq:intro-lambda} always exists in finite dimension.  If $t\mapsto\Phi(\sqrt t)$ is a Young function, Theorem~\ref{thm:lambda-support} gives the implicit consequence
\begin{equation}\label{eq:intro-lambda-up}
 \delta\bigl[\Phi^{-1}(\delta^{-1})\bigr]^2\geq K^{-2},
 \qquad \delta=\frac{|\supp(u)|}{M}.
\end{equation}
For $\Phi(t)=t^p$, $p>2$, this yields a positive-proportion lower bound.  At the logarithmic endpoint $\Phi(t)\simeq t^2\log^\alpha(e+t)$, we combine the inverse identity in Theorem~\ref{thm:lambda-support} with Burstein's $\log^{\alpha/2}\!\log M$ loss~\cite{Burstein}.  This gives the polylogarithmic lower bound stated in the abstract.

Finally, we apply both kinds of uncertainty principles to exact recovery.  Theorem~\ref{thm:l1-recovery} provides an $\ell^1$-minimization criterion under a $(\Phi,\Psi)$ restriction estimate.  Theorem~\ref{thm:random-uniqueness} gives uniqueness of sparse recovery from random Fourier sets satisfying a uniform $\Lambda_\Phi$ inequality.

The organization is as follows.  Section~\ref{sec:preliminaries} contains the standard Orlicz facts and the few technical lemmas used later; each proof is placed immediately after its statement.  Section~\ref{sec:restriction} develops restriction-based uncertainty and annihilating-pair estimates.  Section~\ref{sec:lambda} treats $\Lambda_\Phi$ sets and probabilistic constructions.  Section~\ref{sec:recovery} contains the exact-recovery applications.

\section{Orlicz preliminaries and technical tools}\label{sec:preliminaries}

\subsection{Fourier normalization}

Let $G=\Z_N^d$ and $M=|G|=N^d$.  We identify the dual group with $G$ and write
\[
 \chi(t)=e^{2\pi i t/N}.
\]
For $f:G\to\C$, define
\begin{equation}\label{eq:fourier-transform}
 \widehat f(m)=M^{-1}\sum_{x\in G}\chi(-x\cdot m)f(x).
\end{equation}
Then
\begin{equation}\label{eq:fourier-inversion}
 f(x)=\sum_{m\in G}\chi(x\cdot m)\widehat f(m)
\end{equation}
and
\begin{equation}\label{eq:plancherel}
 \sum_{x\in G}|f(x)|^2=M\sum_{m\in G}|\widehat f(m)|^2.
\end{equation}
For a nonempty set $A\subset G$, $\mu_A$ denotes normalized counting measure on $A$; thus
\[
 \|f\|_{L^p(\mu_A)}=|A|^{-1/p}\|f\|_{L^p(A)}.
\]
We write $\mu=\mu_G$.

The following standard consequence of Fourier inversion is included for reference.

\begin{proposition}[Matolcsi--Sz\H{u}cs; Donoho--Stark]\label{prop:classical-up}
If $f\neq0$, $\supp(f)\subset E$, and $\supp(\widehat f)\subset S$, then
\[
 |E|\,|S|\geq M.
\]
\end{proposition}

\begin{proof}
Fourier inversion and Cauchy--Schwarz give
\[
 \|f\|_\infty\leq |S|^{1/2}\|\widehat f\|_2
 =|S|^{1/2}M^{-1/2}\|f\|_2.
\]
Since $f$ is supported in $E$, $\|f\|_2\leq |E|^{1/2}\|f\|_\infty$.  Cancelling $\|f\|_\infty$ proves the claim.
\end{proof}

\subsection{Young functions and Orlicz norms}

A Young function is a convex, nondecreasing function $\Phi:[0,\infty)\to[0,\infty]$ with $\Phi(0)=0$ and $\Phi(t)\to\infty$ as $t\to\infty$.  We use the generalized inverse
\begin{equation}\label{eq:generalized-inverse}
 \Phi^{-1}(s)=\inf\{t\geq0:\Phi(t)>s\},\qquad s\geq0.
\end{equation}
It is nondecreasing and satisfies
\[
 \Phi(\Phi^{-1}(s))\leq s,
 \qquad
 t\leq\Phi^{-1}(\Phi(t)).
\]
A continuous Young function is called a nice Young function, or an $N$-function, if $\Phi(t)>0$ for $t>0$ and
\[
 \lim_{t\to0^+}\frac{\Phi(t)}t=0,
 \qquad
 \lim_{t\to\infty}\frac{\Phi(t)}t=\infty.
\]
Its complementary Young function is
\[
 \Phi^*(s)=\sup_{t\geq0}\{st-\Phi(t)\}.
\]
Young's inequality states that $st\leq\Phi(t)+\Phi^*(s)$.

We shall also use the standard growth conditions at infinity.  We write $\Phi\in\Delta_2$ if there are $t_0>0$ and $C>0$ such that $\Phi(2t)\leq C\Phi(t)$ for $t\geq t_0$, and $\Phi\in\nabla_2$ if there are $t_0>0$ and $c>1$ such that $\Phi(ct)\geq 2c\Phi(t)$ for $t\geq t_0$.  For $i\in\{a,\infty\}$, set
\begin{align*}
 M_a(r,\Phi)&=\sup_{t>0}\frac{\Phi(rt)}{\Phi(t)},
 &M_\infty(r,\Phi)&=\limsup_{t\to\infty}\frac{\Phi(rt)}{\Phi(t)},\\
 \alpha_\Phi^i&=\lim_{r\to0^+}\frac{\log M_i(r,\Phi)}{\log r},
 &\beta_\Phi^i&=\lim_{r\to\infty}\frac{\log M_i(r,\Phi)}{\log r}.
\end{align*}
These are the Matuszewska--Orlicz indices.  We use $\Phi_1\preceq\Phi_2$ to mean that $\Phi_1(t)\leq C\Phi_2(t)$ for all sufficiently large $t$, with a constant independent of the ambient finite group.

\subsection{Model Young functions}\label{subsec:model-young}

The following examples will be used repeatedly.  We record their inverse growth and indices so that later substitutions are transparent.

\begin{example}[Power functions]\label{ex:young-power}
For $1<p<\infty$, let $\Phi(t)=t^p$.  Up to the conventional constants $1/p$ and $1/p'$, its complementary function is $\Phi^*(t)=t^{p'}$, where $p'=p/(p-1)$.  Moreover,
\[
 \Phi^{-1}(y)=y^{1/p},\qquad
 \alpha_\Phi^a=\beta_\Phi^a=\alpha_\Phi^\infty=\beta_\Phi^\infty=p.
\]
The associated Luxemburg norms are the usual $L^p$ norms, and the normalized-to-counting constants are exactly
\[
 D_\Phi^+(m)=m^{1/p},\qquad D_\Phi^-(m)=m^{-1/p}.
\]
These identities provide a consistency check for every general formula below.
\end{example}

\begin{example}[Power-logarithmic growth]\label{ex:young-power-log}
Fix $p>1$ and $\alpha\in\mathbb R$.  After an equivalent modification on a bounded interval, one may take
\[
 \Phi_{p,\alpha}(t)\simeq t^p\log^\alpha(e+t).
\]
This is a nice Young function satisfying $\Delta_2$, and
\[
 \alpha_{\Phi_{p,\alpha}}^\infty
 =\beta_{\Phi_{p,\alpha}}^\infty=p.
\]
Its inverse satisfies
\begin{equation}\label{eq:power-log-inverse}
 \Phi_{p,\alpha}^{-1}(y)
 \simeq y^{1/p}[\log(e+y)]^{-\alpha/p}
\end{equation}
for large $y$.  To verify \eqref{eq:power-log-inverse}, substitute
\[
 t=y^{1/p}[\log(e+y)]^{-\alpha/p}
\]
into $t^p\log^\alpha(e+t)$ and use $\log(e+t)\simeq\log(e+y)$.  The comparison constants depend only on $p$, $\alpha$, and the chosen modification on the bounded interval.
\end{example}

\begin{example}[The logarithmic endpoint near $L^2$]\label{ex:young-endpoint}
For $\alpha>0$, Burstein's endpoint function is equivalent for large $t$ to
\[
 \Phi_\alpha(t)=t^2\log^\alpha t.
\]
It may be modified for $0\leq t\leq t_0$ so that it is a nice Young function and so that
\[
 \Xi_\alpha(t)=\Phi_\alpha(\sqrt t)
\]
is also a Young function.  The latter condition is exactly the convexity hypothesis needed in Theorem~\ref{thm:lambda-support}.  The inverse asymptotic is
\begin{equation}\label{eq:endpoint-inverse}
 [\Phi_\alpha^{-1}(y)]^2
 \simeq y[\log(e+y)]^{-\alpha}.
\end{equation}
Consequently, the support functional associated with this endpoint is
\begin{equation}\label{eq:endpoint-support-functional}
 \delta[\Phi_\alpha^{-1}(\delta^{-1})]^2
 \simeq[\log(e/\delta)]^{-\alpha}.
\end{equation}
Formula \eqref{eq:endpoint-support-functional} is the reason a logarithmic loss in the $\Lambda_{\Phi_\alpha}$ constant becomes a polylogarithmic loss in spatial support.
\end{example}

For a finite set $A$, the Luxemburg norms with counting and normalized counting measure are
\begin{align}
 \|f\|_{L^\Phi(A)}
 &=\inf\left\{k>0:\sum_{x\in A}\Phi\left(\frac{|f(x)|}{k}\right)\leq1\right\},\label{eq:lux-count}\\
 \|f\|_{L^\Phi(\mu_A)}
 &=\inf\left\{k>0:\frac1{|A|}\sum_{x\in A}\Phi\left(\frac{|f(x)|}{k}\right)\leq1\right\}.\label{eq:lux-normalized}
\end{align}
The associate Orlicz norm is
\begin{equation}\label{eq:associate-norm}
 \|f\|_{L^{(\Phi)}(\mu_A)}
 =\sup\left\{\frac1{|A|}\left|\sum_{x\in A}f(x)g(x)\right|:
 \frac1{|A|}\sum_{x\in A}\Phi^*(|g(x)|)\leq1\right\}.
\end{equation}
The Luxemburg and associate norms are equivalent.  We record the two forms of H\"older's inequality that will be used throughout the paper.

\begin{theorem}[Orlicz H\"older inequalities]\label{thm:orlicz-holder}
Let $\Phi$ be a Young function and let $A$ be a nonempty finite set.  If $f\in L^\Phi(\mu_A)$ and $g\in L^{\Phi^*}(\mu_A)$, then
\begin{align}
 \|fg\|_{L^1(\mu_A)}
 &\leq2\|f\|_{L^\Phi(\mu_A)}\|g\|_{L^{\Phi^*}(\mu_A)},\label{eq:holder-lux}\\
 \|fg\|_{L^1(\mu_A)}
 &\leq\|f\|_{L^{(\Phi)}(\mu_A)}\|g\|_{L^{\Phi^*}(\mu_A)}.\label{eq:holder-associate}
\end{align}
The same statements hold with counting measure in place of normalized counting measure.
\end{theorem}

These are standard results; see Rao--Ren~\cite[Chapters~2--3]{RaoRenTheory} and Maligranda~\cite{Maligranda}.  We use the constant $2$ only when both factors carry Luxemburg norms.  When one factor carries the associate norm, the constant is $1$.  This distinction is responsible for several factors of $2$ in the annihilating-pair estimates below.  Generalized H\"older inequalities of the form used later go back to O'Neil~\cite{ONeil}.

\begin{lemma}[Indicator norms]\label{lem:indicator-norms}
Let $\Phi$ be a Young function and let $E\subset G$ be nonempty.  Then
\begin{align}
 \|\1_E\|_{L^{(\Phi)}(\mu)}
 &=\frac{|E|}{M}(\Phi^*)^{-1}\left(\frac{M}{|E|}\right),\label{eq:indicator-global}\\
 \|\1_E\|_{L^{(\Phi)}(\mu_E)}
 &=(\Phi^*)^{-1}(1),\label{eq:indicator-local}\\
 \|\1_E\|_{L^{(\Phi)}(G)}
 &=|E|(\Phi^*)^{-1}\left(\frac1{|E|}\right).\label{eq:indicator-count}
\end{align}
Moreover,
\begin{equation}\label{eq:indicator-lux}
 \|\1_E\|_{L^\Phi(E)}=\frac1{\Phi^{-1}(|E|^{-1})}.
\end{equation}
\end{lemma}

\begin{proof}
We first prove \eqref{eq:indicator-global}.  Let $g$ satisfy
\[
 \frac1M\sum_{x\in G}\Phi^*(|g(x)|)\leq1.
\]
Jensen's inequality on $E$ gives
\begin{align*}
 \Phi^*\left(\frac1{|E|}\sum_{x\in E}|g(x)|\right)
 &\leq\frac1{|E|}\sum_{x\in E}\Phi^*(|g(x)|)\\
 &\leq\frac{M}{|E|}.
\end{align*}
Applying the generalized inverse and multiplying by $|E|/M$ yields
\[
 \frac1M\sum_{x\in E}|g(x)|
 \leq\frac{|E|}{M}(\Phi^*)^{-1}\left(\frac{M}{|E|}\right).
\]
Taking the supremum over all admissible $g$ proves the upper bound in \eqref{eq:indicator-global}.  For the reverse inequality, set
\[
 g_0=(\Phi^*)^{-1}\left(\frac{M}{|E|}\right)\1_E.
\]
Then
\[
 \frac1M\sum_{x\in G}\Phi^*(|g_0(x)|)=1,
\]
with the usual limiting interpretation for a generalized inverse, and substitution of $g_0$ into the associate norm gives equality.

For \eqref{eq:indicator-local}, the admissibility condition is
\[
 \frac1{|E|}\sum_{x\in E}\Phi^*(|g(x)|)\leq1.
\]
The same Jensen argument gives
\[
 \frac1{|E|}\sum_{x\in E}|g(x)|\leq(\Phi^*)^{-1}(1),
\]
and equality follows by taking $g$ constant on $E$.  With counting measure, the admissibility condition becomes
\[
 \sum_{x\in G}\Phi^*(|g(x)|)\leq1,
\]
which gives \eqref{eq:indicator-count}.  Finally, if $k>0$, then
\[
 \sum_{x\in E}\Phi\left(\frac{\1_E(x)}{k}\right)
 =|E|\Phi(k^{-1}).
\]
The least $k$ for which this is at most $1$ is $1/\Phi^{-1}(|E|^{-1})$, proving \eqref{eq:indicator-lux}.
\end{proof}

\begin{lemma}[Complementary inverse comparison]\label{lem:complementary-inverses}
For a complementary pair of nice Young functions $(\Phi,\Phi^*)$ and every $t>0$,
\begin{equation}\label{eq:complementary-inverses}
 t\leq \Phi^{-1}(t)(\Phi^*)^{-1}(t)\leq2t.
\end{equation}
\end{lemma}

\begin{proof}
This is the standard inverse inequality for complementary Young functions; see Rao--Ren~\cite[Chapter~2]{RaoRenTheory}.  The upper bound follows immediately by applying Young's inequality to $\Phi^{-1}(t)$ and $(\Phi^*)^{-1}(t)$.  The lower bound is the complementary inverse estimate proved there.
\end{proof}

\begin{lemma}[Exact finite-set $L^p$ comparison]\label{lem:exact-lp-comparison}
Let $1\leq p<\infty$, and assume that
\[
 Q(t)=\Phi(t^{1/p})
\]
is a Young function.  If $A$ has cardinality $m$, then
\begin{align}
 \|f\|_{L^p(A)}
 &\leq m^{1/p}\Phi^{-1}(m^{-1})\|f\|_{L^\Phi(A)},\label{eq:exact-lp-count}\\
 \|f\|_{L^p(\mu_A)}
 &\leq \Phi^{-1}(1)\|f\|_{L^\Phi(\mu_A)}.\label{eq:exact-lp-normalized}
\end{align}
\end{lemma}

\begin{proof}
Let $k>\|f\|_{L^\Phi(A)}$.  Since $Q$ is convex, Jensen's inequality gives
\begin{align*}
 Q\left(\frac{1}{m}\sum_{x\in A}\frac{|f(x)|^p}{k^p}\right)
 &\leq \frac{1}{m}\sum_{x\in A}Q\left(\frac{|f(x)|^p}{k^p}\right)\\
 &=\frac{1}{m}\sum_{x\in A}\Phi\left(\frac{|f(x)|}{k}\right)
 \leq \frac{1}{m}.
\end{align*}
Hence
\[
 \frac{1}{m}\sum_{x\in A}|f(x)|^p
 \leq k^pQ^{-1}(m^{-1})
 =k^p[\Phi^{-1}(m^{-1})]^p.
\]
Letting $k$ decrease to the Luxemburg norm proves \eqref{eq:exact-lp-count}.  The same argument with normalized counting measure gives
\[
 Q\left(\frac{1}{m}\sum_{x\in A}\frac{|f(x)|^p}{k^p}\right)\leq1,
\]
and therefore \eqref{eq:exact-lp-normalized}.
\end{proof}

\begin{lemma}[Normalization comparison]\label{lem:normalization-basic}
Let $\Phi$ be a Young function and let $A$ be a nonempty finite set.  Then
\begin{equation}\label{eq:normalization-basic}
 \|f\|_{L^\Phi(A)}\leq |A|\,\|f\|_{L^\Phi(\mu_A)}.
\end{equation}
\end{lemma}

\begin{proof}
Let $k>\|f\|_{L^\Phi(\mu_A)}$.  Convexity and $\Phi(0)=0$ give
\[
 \sum_{x\in A}\Phi\left(\frac{|f(x)|}{|A|k}\right)
 \leq\frac1{|A|}\sum_{x\in A}\Phi\left(\frac{|f(x)|}{k}\right)
 \leq1.
\]
Letting $k$ decrease to the normalized Luxemburg norm proves the claim.
\end{proof}

We shall use the following standard generalized H\"older principle.

\begin{lemma}[O'Neil]\label{lem:generalized-holder}
Let $\Phi,\Phi_1,\Phi_2$ be Young functions satisfying
\[
 \Phi^{-1}(t)\geq\Phi_1^{-1}(t)\Phi_2^{-1}(t),\qquad t\geq0.
\]
Then
\[
 \|fg\|_{L^\Phi(A)}
 \leq2\|f\|_{L^{\Phi_1}(A)}\|g\|_{L^{\Phi_2}(A)}.
\]
The same conclusion holds for normalized counting measure.
\end{lemma}

\begin{proof}
After scaling, assume that the two norms on the right are strictly less than $1$.  The inverse-function condition implies
\[
 \Phi(st)\leq\Phi_1(s)+\Phi_2(t),\qquad s,t\geq0.
\]
Convexity then gives
\[
 \sum_{x\in A}\Phi\left(\frac{|f(x)g(x)|}{2}\right)
 \leq\frac12\sum_{x\in A}\Phi_1(|f(x)|)
   +\frac12\sum_{x\in A}\Phi_2(|g(x)|)
 \leq1.
\]
The general case follows by homogeneity and approximation.
\end{proof}

\begin{theorem}[Bak's Orlicz interpolation theorem \cite{Bak}]\label{thm:Bak}
Let $1\leq r<\infty$, and let $T$ be a sublinear operator from $(X,\mu)$ to $(Y,\nu)$ satisfying
\[
 \nu(\{|Tf|>t\})\leq\left(\frac{A\|f\|_{L^r(d\mu)}}{t}\right)^r,
 \qquad
 \|Tf\|_\infty\leq B\|f\|_\infty.
\]
Write $\Phi(t)=\int_0^t\phi(s)\,ds$, and assume that there are constants $c>1$, $C_0$, and $C_1$ such that
\[
 \int_1^u\frac{\phi(t)}{t^r}\,dt
 \leq C_0\frac{\phi(u)}{u^{r-1}}
 \quad (u>1)
\]
and
\[
 \phi(\lambda)\leq C_1\frac{\phi(\lambda t)}{\phi(t)}
 \quad (\lambda>1,\ t\geq c).
\]
Then
\[
 \|Tf\|_{L^\Phi(d\nu)}
 \leq C_\Phi B\Phi^{-1}\left(\left(\frac AB\right)^r\right)
 \|f\|_{L^\Phi(d\mu)},
\]
where $C_\Phi$ depends only on $r$ and the displayed constants.
\end{theorem}

The theorem is quoted to make explicit the source of the normalization constants used below.  The following consequence places the two measures appearing in this paper side by side.

\begin{corollary}[Counting versus normalized counting]\label{cor:Bak-normalization}
Assume that $\Phi$ satisfies the hypotheses of Theorem~\ref{thm:Bak} with $r=1$.  There are constants $B_\Phi^+$ and $B_\Phi^-$, depending only on $\Phi$, such that for every finite set $A$ of cardinality $m$,
\begin{align}
 \|f\|_{L^\Phi(A)}
 &\leq B_\Phi^+\Phi^{-1}(m)\|f\|_{L^\Phi(\mu_A)},\label{eq:Bak-plus}\\
 \|f\|_{L^\Phi(\mu_A)}
 &\leq B_\Phi^-\Phi^{-1}(m^{-1})\|f\|_{L^\Phi(A)}.\label{eq:Bak-minus}
\end{align}
\end{corollary}

\begin{proof}
For \eqref{eq:Bak-plus}, apply Theorem~\ref{thm:Bak} to the identity map from normalized counting measure on $A$ to counting measure.  It has weak type $(1,1)$ constant $A=m$ and $L^\infty$ constant $B=1$.  Reversing the two measures gives weak type constant $A=m^{-1}$ and proves \eqref{eq:Bak-minus}.
\end{proof}

\begin{lemma}[Elementary $L^p$--Orlicz comparisons]\label{lem:lp-orlicz-comparison}
Let $1\leq p<\infty$, and let $A$ be a nonempty finite set with normalized counting measure.
If $t^p\leq C\Phi(t)$ for all $t\geq0$, then
\[
 \|f\|_{L^p(\mu_A)}\leq C^{1/p}\|f\|_{L^\Phi(\mu_A)}.
\]
If $\Phi(t)\leq C t^p$ for all $t\geq0$, then
\[
 \|f\|_{L^\Phi(\mu_A)}\leq C^{1/p}\|f\|_{L^p(\mu_A)}.
\]
\end{lemma}

\begin{proof}
If $k>\|f\|_{L^\Phi(\mu_A)}$, then
\[
 \frac1{|A|}\sum_{x\in A}\left(\frac{|f(x)|}{k}\right)^p
 \leq C\frac1{|A|}\sum_{x\in A}\Phi\left(\frac{|f(x)|}{k}\right)
 \leq C.
\]
Letting $k$ decrease to the Luxemburg norm proves the first estimate.  For the second, take $k=C^{1/p}\|f\|_{L^p(\mu_A)}$ in the definition of the Luxemburg norm.  \end{proof}

\begin{remark}\label{rem:comparison-hypotheses}
The hypotheses in Lemma~\ref{lem:exact-lp-comparison} and Lemma~\ref{lem:lp-orlicz-comparison} are deliberately stated in different forms.  The first lemma gives an exact finite-set constant, but requires the composition $t\mapsto\Phi(t^{1/p})$ to be convex.  The second lemma needs only a global growth comparison and is therefore more flexible, at the cost of a less explicit constant.  An asymptotic relation such as $t^p\preceq\Phi(t)$ is not, by itself, enough to justify the exact constant $\Phi^{-1}(1)$ unless the behavior of $\Phi$ on a bounded interval is also controlled.  Keeping these two statements separate prevents an implicit dependence on the ambient group.
\end{remark}

\begin{lemma}[Support comparison with $L^2$]\label{lem:orlicz-support-L2}
Let $\Theta$ be a Young function such that $Q(t)=\Theta(\sqrt t)$ is also a Young function.  If $u$ is supported in a set $A$ of cardinality $m$, then
\begin{equation}\label{eq:orlicz-support-L2}
 \|u\|_{L^2(A)}
 \leq m^{1/2}\Theta^{-1}(m^{-1})\|u\|_{L^\Theta(A)}.
\end{equation}
\end{lemma}

\begin{proof}
Let $k>\|u\|_{L^\Theta(A)}$.  Jensen's inequality for the convex function $Q$ gives
\begin{align*}
 Q\left(\frac{1}{m}\sum_{x\in A}\frac{|u(x)|^2}{k^2}\right)
 &\leq\frac{1}{m}\sum_{x\in A}Q\left(\frac{|u(x)|^2}{k^2}\right)\\
 &=\frac{1}{m}\sum_{x\in A}\Theta\left(\frac{|u(x)|}{k}\right)
 \leq\frac{1}{m}.
\end{align*}
Thus
\[
 \|u\|_{L^2(A)}^2
 \leq mk^2Q^{-1}(m^{-1})
 =mk^2[\Theta^{-1}(m^{-1})]^2.
\]
Letting $k$ decrease to the Luxemburg norm and taking square roots proves the result.
\end{proof}

For later statements it is useful to isolate standard norm-comparison constants.  If $\Phi$ grows at least as fast as $t^p$ at infinity, then on probability spaces $L^\Phi$ embeds continuously into $L^p$; if $\Phi$ grows no faster than $t^p$, then $L^p$ embeds continuously into $L^\Phi$.  We denote admissible constants by
\begin{align}
 \|u\|_{L^p(\mu_A)}&\leq a_{\Phi,p}\|u\|_{L^\Phi(\mu_A)},\label{eq:embedding-high}\\
 \|u\|_{L^\Phi(\mu_A)}&\leq b_{\Phi,p}\|u\|_{L^p(\mu_A)}.\label{eq:embedding-low}
\end{align}
More generally, $b_{\Phi,\Theta}$ denotes a constant in
\begin{equation}\label{eq:orlicz-embedding}
 \|u\|_{L^\Phi(\mu_A)}\leq b_{\Phi,\Theta}\|u\|_{L^\Theta(\mu_A)}.
\end{equation}
Only the existence of these constants is used.  Their dependence is exclusively on the displayed Young functions and exponents.  The corresponding embeddings under comparisons valid only for sufficiently large arguments are standard on probability spaces; see Rao--Ren~\cite[Chapter~3]{RaoRenTheory}.

We also need a controlled passage between normalized and counting norms.  For a Young function $\Phi$ and an integer $m\geq1$, let $D_\Phi^+(m)$ and $D_\Phi^-(m)$ denote the best constants for which, on every set $A$ of cardinality $m$,
\begin{align}
 \|u\|_{L^\Phi(A)}
 &\leq D_\Phi^+(m)\|u\|_{L^\Phi(\mu_A)},\label{eq:Dplus}\\
 \|u\|_{L^\Phi(\mu_A)}
 &\leq D_\Phi^-(m)\|u\|_{L^\Phi(A)}.\label{eq:Dminus}
\end{align}
These constants are finite in the present finite-dimensional setting.  Bak's interpolation theorem~\cite{Bak} gives useful quantitative bounds under additional integral hypotheses on the derivative of $\Phi$.  For $\Phi(t)=t^p$, the exact values are $D_\Phi^+(m)=m^{1/p}$ and $D_\Phi^-(m)=m^{-1/p}$.  We retain the symbols $D_\Phi^\pm$ in general rather than imposing a comparison that may depend on the behavior of $\Phi$ near both zero and infinity.

\subsection{Two Fourier-Orlicz estimates}

The next lemma is elementary but central to the first uncertainty argument.

\begin{lemma}[Fourier support and the $L^1$ norm]\label{lem:fourier-support-orlicz}
Let $\Phi$ be a Young function.  If $\supp(\widehat f)\subset S$, then
\begin{equation}\label{eq:fourier-support-orlicz}
 \|f\|_{L^\Phi(G)}
 \leq\frac{|S|}{M\Phi^{-1}(|S|/M)}\|f\|_{L^1(G)}.
\end{equation}
\end{lemma}

\begin{proof}
For a nonzero function $f$, convexity of $\Phi$ implies that $t\mapsto\Phi(t)/t$ is nondecreasing.  It follows directly from the definition of the Luxemburg norm that
\begin{equation}\label{eq:linfty-l1-orlicz}
 \|f\|_{L^\Phi(G)}
 \leq\frac{\|f\|_\infty}
 {\Phi^{-1}(\|f\|_\infty/\|f\|_1)}.
\end{equation}
Fourier inversion and \eqref{eq:fourier-transform} give
\[
 \|f\|_\infty
 \leq |S|\|\widehat f\|_\infty
 \leq\frac{|S|}{M}\|f\|_1.
\]
Since $t/\Phi^{-1}(t)$ is nondecreasing, substituting this bound into \eqref{eq:linfty-l1-orlicz} proves \eqref{eq:fourier-support-orlicz}.
\end{proof}

\begin{lemma}[Spatial support and the $L^1$ norm]\label{lem:l1-spatial-support}
If $\supp(f)\subset E$, then
\begin{equation}\label{eq:l1-spatial-support}
 \|f\|_{L^1(G)}
 \leq |E|\Phi^{-1}(1)\|f\|_{L^\Phi(\mu_E)}.
\end{equation}
Consequently, if also $\supp(\widehat f)\subset S$, then
\begin{equation}\label{eq:combined-support-norm}
 \|f\|_{L^\Phi(G)}
 \leq\frac{|E||S|\Phi^{-1}(1)}{M\Phi^{-1}(|S|/M)}
 \|f\|_{L^\Phi(\mu_E)}.
\end{equation}
\end{lemma}

\begin{proof}
Jensen's inequality gives, for every $k>0$,
\[
 \frac1{|E|}\sum_{x\in E}\Phi\left(\frac{|f(x)|}{k}\right)
 \geq\Phi\left(\frac{\|f\|_1}{k|E|}\right).
\]
The definition of the Luxemburg norm yields \eqref{eq:l1-spatial-support}.  Combining it with Lemma~\ref{lem:fourier-support-orlicz} gives \eqref{eq:combined-support-norm}.
\end{proof}

We finally isolate the precise finite-group Hausdorff--Young estimate needed in the subquadratic arguments.  Rao proved broad Orlicz-space extensions of the Hausdorff--Young theorem~\cite{Rao1968}.  Because the numerical constant depends on the normalization of the Fourier transform, the Haar measures, and the chosen equivalent complementary pair, we formulate the estimate directly in the normalization used throughout this paper.

\begin{definition}[Finite-group Orlicz Hausdorff--Young estimate]\label{def:orlicz-hausdorff-young}
Let $(\Theta,\Theta^*)$ be a complementary pair of Young functions.  We say that the pair satisfies a finite-group Orlicz Hausdorff--Young estimate with constant $H_\Theta(M)$ if
\begin{equation}\label{eq:orlicz-HY}
 \|\widehat f\|_{L^{\Theta^*}(G)}
 \leq H_\Theta(M)\|f\|_{L^\Theta(G)}
\end{equation}
for every $f:G\to\C$, where the Fourier transform is normalized by \eqref{eq:fourier-transform}.  We may take $H_\Theta(M)$ to be the least admissible constant.
\end{definition}

\begin{remark}\label{rem:HY-powers}
For $\Theta(t)=t^q$, $1\leq q\leq2$, the classical Hausdorff--Young inequality in the present normalization gives
\[
 \|\widehat f\|_{L^{q'}(G)}\leq M^{-1/q}\|f\|_{L^q(G)}.
\]
Thus $H_\Theta(M)=M^{-1/q}$ in the power case.  This identity fixes the powers of $M$ in Theorem~\ref{thm:UP-II} and Theorem~\ref{thm:strong-annihilating-pair}.  Rao's theorem and its later variants provide sufficient hypotheses for estimates of the form \eqref{eq:orlicz-HY}; the subsequent results require only the displayed finite-group inequality itself.
\end{remark}

\section{Restriction estimates, uncertainty, and annihilating pairs}\label{sec:restriction}

\subsection{Restriction estimates}

\begin{definition}\label{def:restriction-estimate}
Let $S\subset G$ be nonempty.  Given nice Young functions $\Phi$ and $\Psi$, we say that $S$ satisfies a $(\Phi,\Psi)$ restriction estimate with constant $C_{\Phi,\Psi}(S)$ if
\begin{equation}\label{eq:restriction-estimate}
 \|\widehat g\|_{L^\Psi(\mu_S)}
 \leq C_{\Phi,\Psi}(S)M^{-1}\|g\|_{L^\Phi(G)}
\end{equation}
for every $g:G\to\C$.  The estimate is uniform when $C_{\Phi,\Psi}(S)$ is independent of $N$.
\end{definition}

For $\Phi(t)=t^p$ and $\Psi(t)=t^q$, this is the usual finite $(p,q)$ restriction estimate.  Such estimates were introduced into the finite-field setting by Mockenhaupt and Tao~\cite{MockenhauptTao}; their role in finite uncertainty principles was developed in~\cite{IosevichMayeli}.  We recall the basic support consequence for comparison.

\begin{theorem}[Iosevich--Mayeli \cite{IosevichMayeli}]\label{thm:IM-pq}
Let $1\leq p\leq q\leq\infty$.  Suppose $f\neq0$, $\supp(f)\subset E$, $\supp(\widehat f)\subset S$, and $S$ satisfies a $(p,q)$ restriction estimate with constant $C_{p,q}(S)$.  Then
\[
 |E|^{1/p}|S|\geq \frac{M}{C_{p,q}(S)}.
\]
\end{theorem}

\begin{proof}
Fourier inversion, the embedding $L^q(\mu_S)\hookrightarrow L^1(\mu_S)$, and the restriction estimate give
\[
 \|f\|_\infty\leq |S|\|\widehat f\|_{L^1(\mu_S)}
 \leq C_{p,q}(S)|S|M^{-1}\|f\|_{L^p(G)}.
\]
Since $f$ is supported in $E$, $\|f\|_{L^p(G)}\leq |E|^{1/p}\|f\|_\infty$.  Cancelling the nonzero supremum norm proves the result.
\end{proof}

\subsection{Support uncertainty principles}

\begin{theorem}[Restriction uncertainty I]\label{thm:UP-I}
Suppose that $f\neq0$, $\supp(f)\subset E$, $\supp(\widehat f)\subset S$, and $S$ satisfies a $(\Phi,\Psi)$ restriction estimate.  Then
\begin{equation}\label{eq:UP-I}
 |E|
 \geq
 \Phi^{-1}\left(\frac{|S|}{M}\right)
 \left(\frac{M}{|S|}\right)^2
 \frac1{\Psi^{-1}(1)C_{\Phi,\Psi}(S)}.
\end{equation}
\end{theorem}

\begin{proof}
By Fourier inversion, H\"older's inequality in associate form, and the restriction estimate,
\begin{align*}
 \frac{|f(x)|}{|S|}
 &\leq\|\widehat f\|_{L^1(\mu_S)}\\
 &\leq\|\1_S\|_{L^{(\Psi^*)}(\mu_S)}
       \|\widehat f\|_{L^\Psi(\mu_S)}\\
 &\leq\Psi^{-1}(1)C_{\Phi,\Psi}(S)M^{-1}
       \|f\|_{L^\Phi(G)}.
\end{align*}
Summing over $E$ and applying Lemma~\ref{lem:fourier-support-orlicz},
\[
 \|f\|_1
 \leq |E||S|\Psi^{-1}(1)C_{\Phi,\Psi}(S)M^{-1}
 \frac{|S|}{M\Phi^{-1}(|S|/M)}\|f\|_1.
\]
Since $f\neq0$, cancellation and rearrangement give \eqref{eq:UP-I}.
\end{proof}

\begin{example}\label{ex:UP-I-powers}
If $\Phi(t)=t^p$ and $\Psi(t)=t^q$, then \eqref{eq:UP-I} becomes
\[
 |E|^{1/(2-1/p)}|S|
 \geq \frac{M}{C_{p,q}(S)^{1/(2-1/p)}}.
\]
This first estimate is useful without any comparison of the Orlicz norms with $L^2$.
\end{example}

The next theorem gives the sharper $L^2$-based conclusions.  Its formulation separates the standard norm-comparison constants from the geometry of $E$ and $S$.

\begin{theorem}[Restriction uncertainty II]\label{thm:UP-II}
Suppose that $f\neq0$, $\supp(f)\subset E$, $\supp(\widehat f)\subset S$, and $S$ satisfies a $(\Phi,\Psi)$ restriction estimate with constant $C=C_{\Phi,\Psi}(S)$.

If $L^\Psi$ embeds into $L^2$ with constant $a_{\Psi,2}$ and $L^2$ embeds into $L^\Phi$ with constant $b_{\Phi,2}$, then
\begin{equation}\label{eq:UP-II-high-a}
 |E||S|^3\left[\Phi^{-1}\left(\frac{|S|}{M}\right)\right]^{-2}
 \geq
 \frac{M^3}{\bigl(\Phi^{-1}(1)Ca_{\Psi,2}b_{\Phi,2}\bigr)^2}.
\end{equation}
Moreover, using the normalization constant $D_\Phi^+$ from \eqref{eq:Dplus},
\begin{equation}\label{eq:UP-II-high-b}
 \frac{[D_\Phi^+(|E|)]^2}{|E|}|S|
 \geq\frac{M}{(Ca_{\Psi,2}b_{\Phi,2})^2}.
\end{equation}

Assume instead that the finite-group Orlicz Hausdorff--Young estimate \eqref{eq:orlicz-HY} holds with $\Theta=\Psi$ and constant $H_\Psi(M)$, and that
\[
 \|u\|_{L^\Phi(\mu_A)}
 \leq b_{\Phi,\Psi^*}\|u\|_{L^{\Psi^*}(\mu_A)}.
\]
Then
\begin{equation}\label{eq:UP-II-low}
 D_\Phi^+(|E|)D_{\Psi^*}^-(|E|)D_\Psi^+(|S|)
 \geq\frac1{C H_\Psi(M)b_{\Phi,\Psi^*}}.
\end{equation}
\end{theorem}

\begin{proof}
In the first range, the restriction estimate and the embedding $L^\Psi\hookrightarrow L^2$ give
\begin{equation}\label{eq:L2-restriction}
 \|\widehat f\|_{L^2(\mu_S)}
 \leq Ca_{\Psi,2}M^{-1}\|f\|_{L^\Phi(G)}.
\end{equation}
Since both supports are contained in $E$ and $S$, Plancherel yields
\begin{equation}\label{eq:L2-normalized-support}
 \|\widehat f\|_{L^2(\mu_S)}
 =M^{-1/2}|S|^{-1/2}\|f\|_{L^2(E)}.
\end{equation}
Using \eqref{eq:combined-support-norm} and then $L^2(\mu_E)\hookrightarrow L^\Phi(\mu_E)$, we obtain
\begin{align*}
 M^{-1/2}|S|^{-1/2}\|f\|_2
 &\leq Ca_{\Psi,2}M^{-1}
 \frac{|E||S|\Phi^{-1}(1)}{M\Phi^{-1}(|S|/M)}
 \|f\|_{L^\Phi(\mu_E)}\\
 &\leq Ca_{\Psi,2}b_{\Phi,2}M^{-2}
 \frac{|E|^{1/2}|S|\Phi^{-1}(1)}{\Phi^{-1}(|S|/M)}
 \|f\|_2.
\end{align*}
Cancelling $\|f\|_2$ gives \eqref{eq:UP-II-high-a}.

Alternatively, \eqref{eq:Dplus} gives
\[
 \|f\|_{L^\Phi(G)}
 \leq D_\Phi^+(|E|)\|f\|_{L^\Phi(\mu_E)}
 \leq D_\Phi^+(|E|)b_{\Phi,2}|E|^{-1/2}\|f\|_2.
\]
Substitution into \eqref{eq:L2-restriction}--\eqref{eq:L2-normalized-support} yields \eqref{eq:UP-II-high-b}.

For the subquadratic range, the normalization comparisons and \eqref{eq:orlicz-HY}, applied to $g=\widehat f$, give
\begin{align*}
 \|\widehat f\|_{L^\Psi(\mu_S)}
 &\geq [D_\Psi^+(|S|)]^{-1}\|\widehat f\|_{L^\Psi(G)}\\
 &\geq [D_\Psi^+(|S|)H_\Psi(M)]^{-1}
       \|\widehat{\widehat f}\|_{L^{\Psi^*}(G)}\\
 &=\frac{M^{-1}}{D_\Psi^+(|S|)H_\Psi(M)}
       \|f\|_{L^{\Psi^*}(E)}.
\end{align*}
On the other hand,
\begin{align*}
 \|f\|_{L^\Phi(G)}
 &\leq D_\Phi^+(|E|)\|f\|_{L^\Phi(\mu_E)}\\
 &\leq D_\Phi^+(|E|)b_{\Phi,\Psi^*}
       \|f\|_{L^{\Psi^*}(\mu_E)}\\
 &\leq D_\Phi^+(|E|)b_{\Phi,\Psi^*}D_{\Psi^*}^-(|E|)
       \|f\|_{L^{\Psi^*}(E)}.
\end{align*}
Combining these inequalities with the restriction estimate and cancelling the nonzero norm proves \eqref{eq:UP-II-low}.
\end{proof}

\begin{remark}[Recovery of the Lebesgue-space bounds]\label{rem:recover-lp}
For $\Phi(t)=t^p$ and $\Psi(t)=t^q$, all comparison constants may be taken equal to $1$.  Formula \eqref{eq:UP-II-high-b} gives, for $1\leq p\leq2\leq q$,
\[
 |E|^{(2-p)/p}|S|\geq\frac{M}{C_{p,q}(S)^2}.
\]
Formula \eqref{eq:UP-II-low} gives, for $1\leq p\leq q\leq2$,
\[
 |E|^{(q'-p)q/(q'p)}|S|
 \geq\frac{M}{C_{p,q}(S)^q}.
\]
These are the corresponding estimates of Iosevich and Mayeli~\cite{IosevichMayeli}.
\end{remark}

\subsection{Strong annihilating pairs}

Let $P_E$ denote multiplication by $\1_E$.  Let $Q_S$ be the orthogonal Fourier projection defined by
\[
 \widehat{Q_Sf}=\1_S\widehat f.
\]
Here orthogonality is understood with respect to the unitary Fourier transform $M^{1/2}\widehat f$.  The elementary projection lemma below isolates the functional-analytic mechanism behind every strong annihilating-pair estimate.

\begin{lemma}[Projection criterion]\label{lem:projection-criterion}
If
\begin{equation}\label{eq:projection-rho}
 \|Q_SP_E\|_{L^2(G)\to L^2(G)}\leq\rho<1,
\end{equation}
then for every $f:G\to\C$,
\begin{equation}\label{eq:projection-AP}
 \|f\|_{L^2(G)}
 \leq\left(1+\frac{M^{1/2}}{1-\rho}\right)
 \left(\|f\|_{L^2(E^c)}+\|\widehat f\|_{L^2(S^c)}\right).
\end{equation}
\end{lemma}

\begin{proof}
Decompose $f=P_Ef+P_{E^c}f$ and then decompose the first term according to the Fourier projections.  Since $P_E$ and $Q_S$ are contractions on $L^2$,
\begin{align*}
 \|P_Ef\|_2
 &\leq\|P_EQ_SP_Ef\|_2
      +\|P_EQ_SP_{E^c}f\|_2
      +\|P_EQ_{S^c}f\|_2\\
 &\leq\rho\|P_Ef\|_2
      +\|P_{E^c}f\|_2
      +\|Q_{S^c}f\|_2.
\end{align*}
The unitary form of Plancherel gives
\[
 \|Q_{S^c}f\|_2=M^{1/2}\|\widehat f\|_{L^2(S^c)}.
\]
Hence
\[
 \|P_Ef\|_2
 \leq\frac1{1-\rho}\|f\|_{L^2(E^c)}
 +\frac{M^{1/2}}{1-\rho}\|\widehat f\|_{L^2(S^c)}.
\]
Adding $\|P_{E^c}f\|_2$ and using $M\geq1$ proves \eqref{eq:projection-AP}.
\end{proof}

\begin{definition}
A pair $(E,S)$ of subsets of $G$ is a strong annihilating pair if there is a finite constant $A(E,S)$ such that
\begin{equation}\label{eq:strong-AP-def}
 \|f\|_{L^2(G)}
 \leq A(E,S)
 \left(\|f\|_{L^2(E^c)}+\|\widehat f\|_{L^2(S^c)}\right)
\end{equation}
for every $f:G\to\C$.
\end{definition}

Ghobber and Jaming~\cite{GhobberJaming} proved a quantitative finite-dimensional version of the classical uncertainty principle, and Jaming, Iosevich, and Mayeli~\cite{JamingIosevichMayeli} improved it under $(p,q)$ restriction hypotheses.  The next theorem is the Orlicz analogue.  We state both the superquadratic and subquadratic forms, but avoid repeating the earlier Lebesgue-space theorems verbatim.

\begin{theorem}[Strong annihilating pairs]\label{thm:strong-annihilating-pair}
Assume that $S$ satisfies a $(\Phi,\Psi)$ restriction estimate with constant $C=C_{\Phi,\Psi}(S)$.

Suppose first that $L^\Psi$ embeds into $L^2$ with constant $a_{\Psi,2}$ and that the function $\Theta_\Phi$ determined by
\begin{equation}\label{eq:theta-phi}
 \Theta_\Phi^{-1}(t)=t^{-1/2}\Phi^{-1}(t)
\end{equation}
is a Young function.  Set
\begin{equation}\label{eq:rho-ES}
 \rho_{E,S}=
 \frac{2Ca_{\Psi,2}|S|^{1/2}}
 {M^{1/2}|E|^{1/2}\Phi^{-1}(|E|^{-1})}.
\end{equation}
If $\rho_{E,S}<1$, then
\begin{equation}\label{eq:strong-AP-high}
 \|f\|_2
 \leq\left(1+\frac{M^{1/2}}{1-\rho_{E,S}}\right)
 \left(\|f\|_{L^2(E^c)}+\|\widehat f\|_{L^2(S^c)}\right).
\end{equation}

Suppose instead that the finite-group Orlicz Hausdorff--Young estimate \eqref{eq:orlicz-HY} holds with $\Theta=\Psi$ and constant $H_\Psi(M)$, that $t\mapsto\Psi^*(\sqrt t)$ is a Young function, and that the functions defined by
\[
 \Theta_\Phi^{-1}(t)=t^{-1/2}\Phi^{-1}(t),
 \qquad
 \Theta_\Psi^{-1}(t)=t^{-1/2}\Psi^{-1}(t)
\]
are Young functions.  Define
\begin{equation}\label{eq:Delta-ES}
 \Delta_{E,S}
 =\bigl(H_\Psi(M)(\Psi^*)^{-1}(|E|^{-1})\bigr)^{-1}
 -2C|S|\bigl(\Phi^{-1}(|E|^{-1})\bigr)^{-1}.
\end{equation}
If $S\neq G$ and $\Delta_{E,S}>0$, then
\begin{equation}\label{eq:strong-AP-low}
 \|f\|_2
 \leq (1+A_{E,S})
 \left(\|f\|_{L^2(E^c)}+\|\widehat f\|_{L^2(S^c)}\right),
\end{equation}
where
\begin{equation}\label{eq:A-ES}
 A_{E,S}
 =\frac{2M|E|^{1/2}|S^c|^{-1/2}
 [\Psi^{-1}(|S^c|^{-1})]^{-1}}{\Delta_{E,S}}.
\end{equation}
\end{theorem}

\begin{proof}
Let $f_E=\1_Ef$.  In the first range, the restriction estimate, the embedding into $L^2$, Lemma~\ref{lem:generalized-holder}, and \eqref{eq:indicator-lux} give
\begin{align*}
 \|\widehat{f_E}\|_{L^2(S)}
 &\leq Ca_{\Psi,2}|S|^{1/2}M^{-1}
       \|f_E\|_{L^\Phi(E)}\\
 &\leq2Ca_{\Psi,2}|S|^{1/2}M^{-1}
       \|\1_E\|_{L^{\Theta_\Phi}(E)}\|f\|_{L^2(E)}\\
 &=M^{-1/2}\rho_{E,S}\|f\|_{L^2(E)}.
\end{align*}
Plancherel and the triangle inequality imply
\[
 \|\widehat{f_E}\|_{L^2(S^c)}
 \geq M^{-1/2}(1-\rho_{E,S})\|f\|_{L^2(E)}.
\]
Since $\widehat{f_E}=\widehat f-\widehat{\1_{E^c}f}$,
\[
 M^{-1/2}(1-\rho_{E,S})\|f\|_{L^2(E)}
 \leq\|\widehat f\|_{L^2(S^c)}+M^{-1/2}\|f\|_{L^2(E^c)}.
\]
Adding $\|f\|_{L^2(E^c)}$ and enlarging the common constant gives \eqref{eq:strong-AP-high}.

For the second range, Lemma~\ref{lem:normalization-basic}, generalized H\"older, and the restriction estimate give
\begin{equation}\label{eq:small-S-Psi}
 \|\widehat{f_E}\|_{L^\Psi(S)}
 \leq2C|S|M^{-1}|E|^{-1/2}
 [\Phi^{-1}(|E|^{-1})]^{-1}\|f\|_{L^2(E)}.
\end{equation}
By \eqref{eq:orlicz-HY}, applied to $\widehat{f_E}$, and by Lemma~\ref{lem:orlicz-support-L2},
\begin{equation}\label{eq:global-Psi-lower}
 \|\widehat{f_E}\|_{L^\Psi(G)}
 \geq M^{-1}|E|^{-1/2}
 \bigl(H_\Psi(M)(\Psi^*)^{-1}(|E|^{-1})\bigr)^{-1}
 \|f\|_{L^2(E)}.
\end{equation}
By the triangle inequality for the Luxemburg norm on the disjoint decomposition $G=S\cup S^c$, subtracting \eqref{eq:small-S-Psi} from \eqref{eq:global-Psi-lower} yields
\[
 \|\widehat{f_E}\|_{L^\Psi(S^c)}
 \geq M^{-1}|E|^{-1/2}\Delta_{E,S}\|f\|_{L^2(E)}.
\]
On $S^c$, generalized H\"older with $\Theta_\Psi$ gives
\begin{align*}
 \|\widehat f\|_{L^\Psi(S^c)}
 +\|\widehat{\1_{E^c}f}\|_{L^\Psi(S^c)}
 \leq{}&2|S^c|^{-1/2}[\Psi^{-1}(|S^c|^{-1})]^{-1}\\
 &\times\left(\|\widehat f\|_{L^2(S^c)}+\|f\|_{L^2(E^c)}\right),
\end{align*}
where Plancherel gives
\[
 \|\widehat{\1_{E^c}f}\|_{L^2(S^c)}
 \leq M^{-1/2}\|f\|_{L^2(E^c)}
 \leq\|f\|_{L^2(E^c)}
\]
in the second term.  Combining the last two displays and then adding $\|f\|_{L^2(E^c)}$ proves \eqref{eq:strong-AP-low}.
\end{proof}

\begin{corollary}[Support consequence]\label{cor:annihilating-support}
Let $f\neq0$, $\supp(f)\subset E$, and $\supp(\widehat f)\subset S$.

Under the first set of hypotheses of Theorem~\ref{thm:strong-annihilating-pair},
\begin{equation}\label{eq:AP-support-high}
 |S|\,[\Phi^{-1}(|E|^{-1})]^{-2}|E|^{-1}
 \geq\frac{M}{4C^2a_{\Psi,2}^2}.
\end{equation}
Under the second set of hypotheses, and assuming in addition that $S\neq G$,
\begin{equation}\label{eq:AP-support-low}
 |S|(\Psi^*)^{-1}(|E|^{-1})[\Phi^{-1}(|E|^{-1})]^{-1}
 \geq\frac1{2C H_\Psi(M)}.
\end{equation}
\end{corollary}

\begin{proof}
If the relevant strict reverse inequality held, Theorem~\ref{thm:strong-annihilating-pair} would make the right side of the annihilating-pair estimate vanish, forcing $f=0$.  This contradiction proves both assertions.  The normalization in the second estimate is contained in the factor $H_\Psi(M)$.
\end{proof}

\begin{example}[Lebesgue spaces]\label{ex:AP-powers}
Let $\Phi(t)=t^p$ and $\Psi(t)=t^q$.  In the range $1\leq p\leq2\leq q$, Corollary~\ref{cor:annihilating-support} gives
\[
 |E|^{(2-p)/p}|S|\geq \frac{M}{4C_{p,q}(S)^2}.
\]
In the range $1\leq p\leq q\leq2$, it gives
\[
 |E|^{(q'-p)q/(pq')}|S|^q
 \geq \frac{M}{2^qC_{p,q}(S)^q},
\]
with the normalization in \eqref{eq:orlicz-HY}.  As in the Lebesgue theory, the direct restriction argument in Remark~\ref{rem:recover-lp} gives the sharper support estimate; the annihilating-pair theorem provides the additional stability statement.
\end{example}

\subsection{Worked examples and comparison of the bounds}\label{subsec:worked-examples}

We collect the power-function calculations in one place.  Besides checking the normalizations, these examples clarify which estimates are support inequalities and which provide the stronger stability conclusion of an annihilating pair.

\begin{example}[The range $1\leq p\leq2\leq q$]\label{ex:worked-high-range}
Take $\Phi(t)=t^p$ and $\Psi(t)=t^q$.  The first restriction uncertainty principle gives
\begin{equation}\label{eq:worked-UP-I}
 |E|\geq C_{p,q}(S)^{-1}
 M^{2-1/p}|S|^{-2+1/p},
\end{equation}
or equivalently
\begin{equation}\label{eq:worked-UP-I-rearranged}
 |E|^{1/(2-1/p)}|S|
 \geq M C_{p,q}(S)^{-1/(2-1/p)}.
\end{equation}
The calculation uses only
\[
 \Phi^{-1}(|S|/M)=(|S|/M)^{1/p}
 \quad\text{and}\quad
 \Psi^{-1}(1)=1.
\]

The sharper $L^2$ argument uses
\[
 D_\Phi^+(|E|)=|E|^{1/p},
 \qquad a_{\Psi,2}=b_{\Phi,2}=1,
\]
and therefore gives
\begin{equation}\label{eq:worked-UP-II-high}
 |E|^{(2-p)/p}|S|
 \geq \frac{M}{C_{p,q}(S)^2}.
\end{equation}
For $p<2$, this is the bound obtained in the Lebesgue-space theory of Iosevich and Mayeli.  The same support exponent follows from the superquadratic annihilating-pair theorem, with the constant $4C_{p,q}(S)^2$ in place of $C_{p,q}(S)^2$.  The loss of $4$ is the price of proving the stable estimate \eqref{eq:strong-AP-high}, which remains meaningful when the signal and its Fourier transform are merely concentrated, rather than exactly supported, on $E$ and $S$.
\end{example}

\begin{example}[The range $1\leq p\leq q\leq2$]\label{ex:worked-low-range}
Again take $\Phi(t)=t^p$ and $\Psi(t)=t^q$.  Now $\Psi^*(t)$ has power $q'=q/(q-1)$, and
\[
 D_\Phi^+(|E|)=|E|^{1/p},\qquad
 D_{\Psi^*}^-(|E|)=|E|^{-1/q'},\qquad
 D_\Psi^+(|S|)=|S|^{1/q}.
\]
Since $H_\Psi(M)=M^{-1/q}$, Theorem~\ref{thm:UP-II} becomes
\[
 |E|^{1/p-1/q'}|S|^{1/q}
 \geq \frac{M^{1/q}}{C_{p,q}(S)}.
\]
Raising this inequality to the $q$th power gives
\begin{equation}\label{eq:worked-UP-II-low}
 |E|^{q/p-q/q'}|S|
 \geq \frac{M}{C_{p,q}(S)^q}.
\end{equation}
The exponent may also be written as
\[
 \frac{(q'-p)q}{q'p}.
\]
The subquadratic annihilating-pair argument yields
\begin{equation}\label{eq:worked-AP-low}
 |E|^{q/p-q/q'}|S|^q
 \geq \frac{M}{2^qC_{p,q}(S)^q}.
\end{equation}
Although \eqref{eq:worked-AP-low} is not as sharp as \eqref{eq:worked-UP-II-low} as a pure support estimate, it is accompanied by the quantitative stability statement \eqref{eq:strong-AP-low}.
\end{example}

\begin{example}[A superquadratic logarithmic perturbation]\label{ex:worked-power-log}
Let
\[
 \Phi(t)\simeq t^p\log^\alpha(e+t),\qquad p>2.
\]
Then
\[
 \Phi^{-1}(y)\simeq y^{1/p}[\log(e+y)]^{-\alpha/p}.
\]
If $S$ is a uniform $\Lambda_\Phi$ set with constant $K$, Theorem~\ref{thm:lambda-support} gives
\[
 \delta^{1-2/p}
 [\log(e/\delta)]^{-2\alpha/p}
 \gtrsim_{p,\alpha}K^{-2},
 \qquad \delta=\frac{|\supp(u)|}{M}.
\]
Equivalently,
\[
 \delta^{1-2/p}
 \gtrsim_{p,\alpha}K^{-2}
 [\log(e/\delta)]^{2\alpha/p}.
\]
This calculation explains why the inverse function is squared in the correct $\Lambda_\Phi$ uncertainty formula and why the logarithmic exponent is $2\alpha/p$, rather than $\alpha/p$.
\end{example}

\section{$\Lambda_\Phi$ sets and random Fourier support}\label{sec:lambda}

\subsection{Classical background}

For $p>2$, a set $S\subset G$ is called a $\Lambda_p$ set if there is a constant $K_p(S)$ such that
\[
 \|u\|_{L^p(\mu)}\leq K_p(S)\|u\|_{L^2(\mu)}
\]
for every $u$ whose Fourier transform is supported in $S$.  Rudin initiated the systematic study of such sets~\cite{Rudin}.  The probabilistic theorem of Bourgain provides the model for the Orlicz results used below.

\begin{theorem}[Bourgain \cite{Bourgain}]\label{thm:Bourgain}
Let $(\varphi_j)_{j=1}^M$ be an orthonormal system on a probability space and assume that $\|\varphi_j\|_\infty\leq1$.  For every $p>2$, there are constants $c_p,C_p>0$, depending only on $p$, and a subset $S\subset\{1,\dots,M\}$ with
\[
 |S|\geq c_pM^{2/p}
\]
such that
\[
 \left\|\sum_{j\in S}a_j\varphi_j\right\|_{L^p}
 \leq C_p\left(\sum_{j\in S}|a_j|^2\right)^{1/2}
\]
for every coefficient sequence $(a_j)_{j\in S}$.
\end{theorem}

For the character system of $G$, the coefficient $\ell^2$ norm in Theorem~\ref{thm:Bourgain} is exactly the spatial $L^2(\mu)$ norm of the corresponding trigonometric polynomial.  Thus the theorem produces uniform $\Lambda_p$ sets of essentially maximal possible size.  The Orlicz-space results below replace the power $t^p$ by more general Young functions and track how the admissible density depends on their inverse functions.

\subsection{Uniform $\Lambda_\Phi$ inequalities}

\begin{definition}\label{def:lambda-phi}
Let $\Phi$ be a nice Young function and $S\subset G$.  We call $S$ a uniform $\Lambda_\Phi$ set with constant $K$ if
\begin{equation}\label{eq:lambda-phi-def}
 \|u\|_{L^\Phi(\mu)}\leq K\|u\|_{L^2(\mu)}
\end{equation}
for every $u:G\to\C$ such that $\supp(\widehat u)\subset S$.
\end{definition}

The adjective ``uniform'' emphasizes that the same $K$ works for the entire Fourier-supported subspace.  In a fixed finite-dimensional space, norm equivalence gives a constant for each set $S$, but probabilistic $\Lambda_\Phi$ theorems are significant because $K$ is controlled independently of the ambient cardinality.

\begin{theorem}[$\Lambda_\Phi$ support transfer]\label{thm:lambda-support}
Let $S$ be a uniform $\Lambda_\Phi$ set with constant $K$.  Assume that
\begin{equation}\label{eq:Xi-definition}
 \Xi(t)=\Phi(\sqrt t)
\end{equation}
is a Young function.  If $u\neq0$, $\supp(\widehat u)\subset S$, and $E=\supp(u)$, then, with $\delta=|E|/M$,
\begin{equation}\label{eq:lambda-support-main}
 \delta\bigl[\Phi^{-1}(\delta^{-1})\bigr]^2\geq K^{-2}.
\end{equation}
\end{theorem}

\begin{proof}
By construction of $\Xi$,
\begin{equation}\label{eq:square-norm}
 \||u|^2\|_{L^\Xi(\mu)}=\|u\|_{L^\Phi(\mu)}^2.
\end{equation}
H\"older's inequality in associate form and Lemma~\ref{lem:indicator-norms} give
\begin{align*}
 \|u\|_{L^2(\mu)}^2
 &=\||u|^2\1_E\|_{L^1(\mu)}\\
 &\leq\||u|^2\|_{L^\Xi(\mu)}
       \|\1_E\|_{L^{(\Xi^*)}(\mu)}\\
 &=\|u\|_{L^\Phi(\mu)}^2
   \delta\Xi^{-1}(\delta^{-1}).
\end{align*}
Since $\Xi(t)=\Phi(\sqrt t)$,
\begin{equation}\label{eq:correct-Xi-inverse}
 \Xi^{-1}(y)=\bigl[\Phi^{-1}(y)\bigr]^2.
\end{equation}
Applying the uniform $\Lambda_\Phi$ inequality and cancelling $\|u\|_2^2$ proves \eqref{eq:lambda-support-main}.
\end{proof}

\begin{remark}\label{rem:lambda-formulation}
The uniform hypothesis in Definition~\ref{def:lambda-phi} cannot be replaced by an inequality asserted only for the particular function $u$ appearing in the theorem: such an inequality is automatic with a function-dependent constant.  The content is that $K$ depends only on the Fourier set and the Young function.
\end{remark}

\begin{corollary}[Power and power-logarithmic growth]\label{cor:power-log-support}
Suppose $S$ is a uniform $\Lambda_\Phi$ set with constant $K$, and assume that $t\mapsto\Phi(\sqrt t)$ is a Young function.

If $\Phi(t)=t^p$, $p>2$, then every nonzero $u$ with $\supp(\widehat u)\subset S$ satisfies
\begin{equation}\label{eq:power-support}
 |\supp(u)|\geq M K^{-2p/(p-2)}.
\end{equation}

If, for large $t$,
\[
 \Phi(t)\simeq t^p\log^\alpha(e+t),\qquad p>2,
\]
then, for $\delta=|\supp(u)|/M$ sufficiently small,
\begin{equation}\label{eq:power-log-support}
 \delta^{1-2/p}
 \gtrsim_{p,\alpha} K^{-2}
 \bigl[\log(e/\delta)\bigr]^{2\alpha/p}.
\end{equation}
\end{corollary}

\begin{proof}
For powers, \eqref{eq:lambda-support-main} reads $\delta^{1-2/p}\geq K^{-2}$.  For the second assertion, the standard inverse asymptotic
\[
 \Phi^{-1}(y)\simeq y^{1/p}[\log(e+y)]^{-\alpha/p}
\]
inserted into \eqref{eq:lambda-support-main} gives \eqref{eq:power-log-support}.  The logarithmic exponent is $2\alpha/p$ because the inverse is squared in \eqref{eq:lambda-support-main}.
\end{proof}

\subsection{Probabilistic constructions}

We use Bernoulli selection throughout this subsection.  Given an orthonormal system indexed by $\{1,\dots,M\}$ and a parameter $0<\delta<1$, let $(\xi_j)_{j=1}^M$ be independent random variables with
\[
 \mathbb P(\xi_j=1)=\delta,
 \qquad
 \mathbb P(\xi_j=0)=1-\delta,
\]
and set
\[
 S_\omega=\{j:\xi_j(\omega)=1\}.
\]
Then $\mathbb E|S_\omega|=\delta M$.  The cited probabilistic theorems control the operator norm of the selected system, usually in expectation or with an explicit positive probability.  A concentration inequality for $|S_\omega|$, combined with Markov's inequality for the operator norm when necessary, produces one event on which both the cardinality and the uniform $\Lambda_\Phi$ estimate hold.  We never interpret these results as statements about a uniformly chosen subset of an exact cardinality unless the cited theorem itself uses that model.

For completeness, recall that $\Phi\in\Delta_2$ if there are $t_0>0$ and $C>0$ such that $\Phi(2t)\leq C\Phi(t)$ for $t\geq t_0$.  Its lower Matuszewska--Orlicz index at infinity is
\[
 \alpha_\Phi^\infty
 =\lim_{r\to0^+}\frac{\log M_\infty(r,\Phi)}{\log r},
 \qquad
 M_\infty(r,\Phi)=\limsup_{t\to\infty}\frac{\Phi(rt)}{\Phi(t)}.
\]

Rudin introduced $\Lambda_p$ sets~\cite{Rudin}, and Bourgain proved that every uniformly bounded orthogonal system of length $M$ contains a $\Lambda_p$ subsystem of cardinality comparable to $M^{2/p}$ with a constant depending only on $p$~\cite{Bourgain}.  Ryou extended this probabilistic mechanism to a broad class of Orlicz spaces~\cite{Ryou}.  We use the following consequence of his theorem.

\begin{theorem}[Ryou, specialized form]\label{thm:Ryou-specialized}
Let $\Phi$ be a nice Young function in $\Delta_2$ whose lower Matuszewska--Orlicz index at infinity is strictly greater than $2$.  For all sufficiently large $M$, a Bernoulli selection from the character system on $G$, with mean cardinality $[\Phi^{-1}(M)]^2$, has a realization $S$ for which
\begin{equation}\label{eq:Ryou-lambda}
 \|u\|_{L^\Phi(\mu)}\leq K_\Phi\|u\|_{L^2(\mu)}
\end{equation}
for all $u$ with $\supp(\widehat u)\subset S$, where $K_\Phi$ is independent of $M$.  More precisely, the event that both $|S|\asymp[\Phi^{-1}(M)]^2$ and \eqref{eq:Ryou-lambda} hold has probability bounded below by a positive constant depending only on $\Phi$.
\end{theorem}

Indeed, apply Theorem~1.5 and Remark~1.7(3) of~\cite{Ryou} with the auxiliary Young function $\Phi_0(t)=t^2$ and with the $M$ characters indexed by $\{1,\dots,M\}$.  Orthonormality gives the required $L^2$ operator bound.  Ryou's expected-norm estimate, Markov's inequality, and concentration of the Bernoulli cardinality then give the stated positive-probability event, with constants depending only on $\Phi$.

\begin{remark}\label{rem:Ryou-Limonova}
Ryou's theorem is formulated for Bernoulli subsets and expected operator norms, rather than for a uniformly chosen subset of an exact prescribed cardinality.  Markov's inequality, together with the usual concentration estimate for the cardinality of a Bernoulli set, gives the positive-probability formulation above.  Limonova~\cite{Limonova} obtained complementary results for logarithmic Orlicz growth.  Her estimates may be inserted directly into Theorem~\ref{thm:lambda-support}; we do not state a separate corollary because Burstein's theorem gives the sharper endpoint loss for the uniformly bounded orthogonal systems considered here.
\end{remark}

\begin{corollary}[Random uncertainty from Ryou]\label{cor:Ryou-random-uncertainty}
Under the hypotheses of Theorem~\ref{thm:Ryou-specialized}, assume in addition that $t\mapsto\Phi(\sqrt t)$ is a Young function.  With probability bounded below by a positive constant depending only on $\Phi$, the selected set $S$ has cardinality
\[
 |S|\asymp[\Phi^{-1}(M)]^2
\]
and every nonzero $u$ with $\supp(\widehat u)\subset S$ satisfies
\begin{equation}\label{eq:Ryou-support-implicit}
 \frac{|\supp(u)|}{M}
 \left[\Phi^{-1}\left(\frac{M}{|\supp(u)|}\right)\right]^2
 \geq K_\Phi^{-2}.
\end{equation}
\end{corollary}

\begin{proof}
Theorem~\ref{thm:Ryou-specialized} supplies, on a single event of positive probability, both the cardinality estimate and the uniform inequality \eqref{eq:Ryou-lambda}. Apply Theorem~\ref{thm:lambda-support} on that event.  Because the $\Lambda_\Phi$ estimate is uniform over the entire Fourier-supported subspace, the conclusion holds simultaneously for every nonzero $u$ with Fourier support in $S$.
\end{proof}

\begin{example}[Power and power-logarithmic Ryou sets]\label{ex:Ryou-power-log}
For $\Phi(t)=t^p$, $p>2$, Theorem~\ref{thm:Ryou-specialized} produces a set of cardinality comparable to $M^{2/p}$, and Corollary~\ref{cor:Ryou-random-uncertainty} gives
\[
 |\supp(u)|\geq c_pM
\]
for every nonzero $u$ whose Fourier transform is supported in that set.  This recovers the familiar positive-proportion uncertainty consequence of Bourgain's theorem.

More generally, suppose
\[
 \Phi(t)\simeq t^p\log^\alpha(e+t),\qquad p>2.
\]
Then
\[
 [\Phi^{-1}(M)]^2
 \simeq M^{2/p}[\log(e+M)]^{-2\alpha/p},
\]
so Ryou's theorem selects Fourier sets of this order.  The corresponding spatial support is determined by
\[
 \delta^{1-2/p}
 \gtrsim_{p,\alpha}K_\Phi^{-2}
 [\log(e/\delta)]^{2\alpha/p}.
\]
In particular, because $K_\Phi$ is independent of $M$, the relative support $\delta$ is bounded below by a positive constant depending only on $p$, $\alpha$, and the chosen equivalent Young function.
\end{example}

At the endpoint of quadratic growth with a logarithmic correction, we use Burstein's endpoint theorem.

\begin{theorem}[Burstein]\label{thm:Burstein}
Fix $\alpha>0$, and let $\Phi_\alpha$ be a nice Young function equivalent for large $t$ to
\[
 t^2\log^\alpha t.
\]
For all sufficiently large $M$, let $S$ be the Bernoulli random subset of the $M$ characters of $G$ furnished by Burstein's construction.  With probability at least $1/4$, both
\begin{equation}\label{eq:Burstein-size}
 |S|\geq\frac{M}{e\log^{\alpha+1}M}
\end{equation}
and
\begin{equation}\label{eq:Burstein-lambda}
 \|u\|_{L^{\Phi_\alpha}(\mu)}
 \leq K_\alpha\log^{\alpha/2}(\log M)\|u\|_{L^2(\mu)}
\end{equation}
hold for every $u$ with $\supp(\widehat u)\subset S$.
\end{theorem}

\begin{corollary}[Random endpoint uncertainty]\label{cor:Burstein-support}
After an equivalent modification on a bounded interval, choose $\Phi_\alpha$ so that $t\mapsto\Phi_\alpha(\sqrt t)$ is a Young function.  On the event in Theorem~\ref{thm:Burstein}, every nonzero $u$ satisfying $\supp(\widehat u)\subset S$ obeys
\begin{equation}\label{eq:Burstein-support}
 |\supp(u)|\geq \frac{M}{(\log M)^{C_\alpha}},
\end{equation}
where $C_\alpha>0$ depends only on $\alpha$.
\end{corollary}

\begin{proof}
For large $y$,
\[
 [\Phi_\alpha^{-1}(y)]^2\simeq y[\log(e+y)]^{-\alpha}.
\]
Let $\delta=|\supp(u)|/M$.  Theorem~\ref{thm:lambda-support} and \eqref{eq:Burstein-lambda} imply
\[
 [\log(e/\delta)]^{-\alpha}
 \gtrsim_\alpha
 \bigl(K_\alpha\log^{\alpha/2}(\log M)\bigr)^{-2}
 \simeq_\alpha [\log\log M]^{-\alpha}.
\]
Consequently,
\[
 \log(e/\delta)\lesssim_\alpha\log\log M.
\]
Exponentiating and changing the constant gives \eqref{eq:Burstein-support}.
\end{proof}

\subsection{From $\Lambda_\Phi$ to restriction}

\begin{theorem}[$\Lambda_\Phi$ implies a restriction estimate]\label{thm:lambda-to-restriction}
Suppose that $S$ is a uniform $\Lambda_\Phi$ set with constant $K$.  Then every $f:G\to\C$ satisfies
\begin{equation}\label{eq:lambda-to-restriction}
 \|\widehat f\|_{L^2(\mu_S)}
 \leq
 \frac{2K D_\Phi^+(M)}{|S|^{1/2}}M^{-1}
 \|f\|_{L^{\Phi^*}(G)}.
\end{equation}
Consequently, \eqref{eq:lambda-to-restriction} is a uniform $(\Phi^*,2)$ restriction estimate provided that
\[
 D_\Phi^+(M)\lesssim_\Phi |S|^{1/2}.
\]
\end{theorem}

\begin{proof}
Let $v$ be defined by
\[
 \widehat v(m)=\widehat f(m)\1_S(m).
\]
Then $\supp(\widehat v)\subset S$.  By the polarization form of Plancherel,
\begin{align*}
 \|\widehat f\|_{L^2(\mu_S)}^2
 &=\frac1{|S|}\sum_{m\in S}\widehat f(m)\overline{\widehat v(m)}\\
 &=\frac1{M|S|}\sum_{x\in G}f(x)\overline{v(x)}.
\end{align*}
H\"older's inequality, the normalization comparison \eqref{eq:Dplus}, and the uniform $\Lambda_\Phi$ inequality give
\begin{align*}
 \|\widehat f\|_{L^2(\mu_S)}^2
 &\leq\frac{2}{M|S|}\|f\|_{L^{\Phi^*}(G)}\|v\|_{L^\Phi(G)}\\
 &\leq\frac{2D_\Phi^+(M)}{M|S|}
       \|f\|_{L^{\Phi^*}(G)}\|v\|_{L^\Phi(\mu)}\\
 &\leq\frac{2K D_\Phi^+(M)}{M|S|}
       \|f\|_{L^{\Phi^*}(G)}\|v\|_{L^2(\mu)}.
\end{align*}
Plancherel gives
\[
 \|v\|_{L^2(\mu)}
 =\left(\sum_{m\in S}|\widehat f(m)|^2\right)^{1/2}
 =|S|^{1/2}\|\widehat f\|_{L^2(\mu_S)}.
\]
Cancelling the last factor proves \eqref{eq:lambda-to-restriction}.  The final assertion is immediate from \eqref{eq:lambda-to-restriction}.
\end{proof}

\section{Exact recovery}\label{sec:recovery}

Suppose that the Fourier coefficients of an unknown signal are observed on $S^c$ and missing on $S$.  There are two distinct recovery questions, and it is useful to keep them separate.

First, one may assume in advance that both the unknown signal and every admissible competitor are $s$-sparse.  In that case a support uncertainty principle gives uniqueness directly: the difference of two admissible reconstructions has Fourier support in $S$ and spatial support of size at most $2s$.  We refer to this as sparse uniqueness.

Second, one may minimize a convex functional over all signals agreeing with the observed data, without imposing a sparsity constraint on the competitors.  For $\ell^1$ minimization, uniqueness requires the null-space inequality
\begin{equation}\label{eq:NSP-intro}
 \|h\|_{L^1(E)}<\|h\|_{L^1(E^c)}
\end{equation}
for every nonzero $h$ whose Fourier transform is supported in $S$.  A support uncertainty principle alone does not imply \eqref{eq:NSP-intro}; the proof of Theorem~\ref{thm:l1-recovery} establishes it quantitatively.  This distinction corrects a common ambiguity in formulations of exact recovery on finite groups.

We begin with the direct sparse-uniqueness consequence of the support uncertainty principle.

\begin{corollary}[Sparse uniqueness via restriction]\label{cor:sparse-uniqueness-restriction}
Suppose $S$ satisfies a $(\Phi,\Psi)$ restriction estimate.  Let
\begin{equation}\label{eq:L-S}
 L(S)=
 \Phi^{-1}\left(\frac{|S|}{M}\right)
 \left(\frac{M}{|S|}\right)^2
 \frac1{\Psi^{-1}(1)C_{\Phi,\Psi}(S)}.
\end{equation}
If $2s<L(S)$, then every signal supported on at most $s$ points is uniquely determined, among all signals supported on at most $s$ points, by its Fourier coefficients on $S^c$.
\end{corollary}

\begin{proof}
If two such signals $f$ and $g$ have the same observed Fourier coefficients, then $h=f-g$ satisfies $|\supp(h)|\leq2s$ and $\supp(\widehat h)\subset S$.  Theorem~\ref{thm:UP-I} forces either $h=0$ or $|\supp(h)|\geq L(S)$.  The strict inequality $2s<L(S)$ excludes the second possibility.
\end{proof}

\begin{example}[Sparse uniqueness in the power case]\label{ex:sparse-uniqueness-power}
Let $\Phi(t)=t^p$ and $\Psi(t)=t^q$.  Corollary~\ref{cor:sparse-uniqueness-restriction} gives uniqueness among $s$-sparse signals whenever
\begin{equation}\label{eq:sparse-uniqueness-power}
 2s<\frac1{C_{p,q}(S)}
 \left(\frac{M}{|S|}\right)^{2-1/p}.
\end{equation}
This follows from
\[
 \Phi^{-1}(|S|/M)=(|S|/M)^{1/p}.
\]
The condition is obtained from the first uncertainty principle and is valid without an $L^2$ embedding assumption.  In ranges where Theorem~\ref{thm:UP-II} applies, its sharper support bound can be used instead and gives the corresponding improved sparse-uniqueness threshold.
\end{example}

We next give a sufficient null-space condition for recovery by $\ell^1$ minimization.

\begin{theorem}[$\ell^1$ recovery via restriction]\label{thm:l1-recovery}
Assume that $S$ satisfies a $(\Phi,\Psi)$ restriction estimate with constant $C=C_{\Phi,\Psi}(S)$ and that
\[
 \|u\|_{L^2(\mu_S)}\leq a_{\Psi,2}\|u\|_{L^\Psi(\mu_S)}.
\]
Let $f$ be supported in $E$.  If
\begin{equation}\label{eq:l1-recovery-condition}
 |E|<
 \frac{M^3[\Phi^{-1}(|S|/M)]^2}
 {4C^2a_{\Psi,2}^2|S|^3},
\end{equation}
then $f$ is the unique minimizer of
\begin{equation}\label{eq:l1-program}
 \min_{u:G\to\C}\|u\|_{L^1(G)}
 \quad\text{subject to}\quad
 \widehat u(m)=\widehat f(m)\quad(m\in S^c).
\end{equation}
\end{theorem}

\begin{proof}
Let $h\neq0$ satisfy $\supp(\widehat h)\subset S$.  By Cauchy--Schwarz, Plancherel, the norm embedding, the restriction estimate, and Lemma~\ref{lem:fourier-support-orlicz},
\begin{align*}
 \|h\|_{L^1(E)}
 &\leq |E|^{1/2}\|h\|_2\\
 &=|E|^{1/2}M^{1/2}|S|^{1/2}
   \|\widehat h\|_{L^2(\mu_S)}\\
 &\leq Ca_{\Psi,2}|E|^{1/2}|S|^{1/2}M^{-1/2}
   \|h\|_{L^\Phi(G)}\\
 &\leq
 \frac{Ca_{\Psi,2}|E|^{1/2}|S|^{3/2}}
 {M^{3/2}\Phi^{-1}(|S|/M)}\|h\|_1.
\end{align*}
Condition \eqref{eq:l1-recovery-condition} makes the last coefficient strictly less than $1/2$.  Therefore
\begin{equation}\label{eq:null-space-property}
 \|h\|_{L^1(E)}<\|h\|_{L^1(E^c)}
\end{equation}
for every nonzero $h$ in the Fourier null space.

If $g$ is any admissible competitor in \eqref{eq:l1-program}, write $h=f-g$.  Then $\supp(\widehat h)\subset S$, and
\begin{align*}
 \|g\|_1
 &=\|f-h\|_{L^1(E)}+\|h\|_{L^1(E^c)}\\
 &\geq\|f\|_1-\|h\|_{L^1(E)}+\|h\|_{L^1(E^c)}.
\end{align*}
If $h\neq0$, \eqref{eq:null-space-property} shows that $\|g\|_1>\|f\|_1$.  Hence $h=0$ for every minimizer, and $f$ is recovered uniquely.
\end{proof}

\begin{example}[Lebesgue-space recovery]\label{ex:l1-recovery-powers}
Let $\Phi(t)=t^p$ and $\Psi(t)=t^q$ with $1\leq p\leq2\leq q$, and suppose that $S$ satisfies the corresponding $(p,q)$ restriction estimate with constant $C_{p,q}(S)$.  Since $a_{\Psi,2}=1$ in this range, condition~\eqref{eq:l1-recovery-condition} becomes
\[
 |E|<\frac{M^{(3p-2)/p}}
 {4C_{p,q}(S)^2|S|^{(3p-2)/p}}.
\]
Thus the Orlicz theorem recovers the usual $\ell^1$-recovery criterion in the Lebesgue setting.  Equivalently,
\[
 |E|<\frac1{4C_{p,q}(S)^2}
 \left(\frac{M}{|S|}\right)^{3-2/p}.
\]
The exponent $3-2/p$ is larger than the exponent $2-1/p$ in the sparse-uniqueness condition \eqref{eq:sparse-uniqueness-power}, but the constants and the interpretation are different.  The present condition certifies the minimizer among all admissible signals, whereas \eqref{eq:sparse-uniqueness-power} compares only signals satisfying the same sparsity constraint.
\end{example}

The $\Lambda_\Phi$ support theorem also gives a clean random sparse-uniqueness criterion.  We first record the monotonicity needed to pass from the support size of a difference to the prescribed sparsity threshold.

\begin{lemma}[Monotonicity of the support functional]\label{lem:support-functional-monotone}
Assume that $\Xi(t)=\Phi(\sqrt t)$ is a Young function.  Then
\[
 F_\Phi(\delta)=\delta[\Phi^{-1}(\delta^{-1})]^2,
 \qquad 0<\delta\leq1,
\]
is nondecreasing.
\end{lemma}

\begin{proof}
By \eqref{eq:correct-Xi-inverse},
\[
 F_\Phi(\delta)=\delta\Xi^{-1}(\delta^{-1}).
\]
Since $\Xi$ is convex and $\Xi(0)=0$, the function $t\mapsto\Xi(t)/t$ is nondecreasing on $(0,\infty)$.  Consequently $y\mapsto\Xi^{-1}(y)/y$ is nonincreasing.  Substituting $y=\delta^{-1}$ shows that $\delta\mapsto\delta\Xi^{-1}(\delta^{-1})$ is nondecreasing.
\end{proof}

\begin{theorem}[Sparse uniqueness from a $\Lambda_\Phi$ set]\label{thm:random-uniqueness}
Let $S$ be a uniform $\Lambda_\Phi$ set with constant $K$, and assume that $\Xi(t)=\Phi(\sqrt t)$ is a Young function.  Let $s$ satisfy $1\leq 2s\leq M$.  If
\begin{equation}\label{eq:random-uniqueness-condition}
 \frac{2s}{M}
 \left[\Phi^{-1}\left(\frac{M}{2s}\right)\right]^2
 <K^{-2},
\end{equation}
then every signal supported on at most $s$ points is uniquely determined, among all signals supported on at most $s$ points, by its Fourier coefficients on $S^c$.

For $\Phi(t)=t^p$, $p>2$, condition \eqref{eq:random-uniqueness-condition} becomes
\begin{equation}\label{eq:random-uniqueness-power}
 s<\frac{M}{2K^{2p/(p-2)}}.
\end{equation}
\end{theorem}

\begin{proof}
If two $s$-sparse signals have the same observed Fourier coefficients, their difference $h$ has Fourier support in $S$ and spatial support of size at most $2s$.  By Lemma~\ref{lem:support-functional-monotone}, the left-hand side of the support inequality for $h$ is at most the left-hand side of \eqref{eq:random-uniqueness-condition}.  Thus any nonzero such $h$ would contradict Theorem~\ref{thm:lambda-support}.  Hence $h=0$.  The power-function calculation is immediate.
\end{proof}

\begin{corollary}[Concrete random sparse-uniqueness regimes]\label{cor:random-recovery-regimes}
The following conclusions hold.

For every $p>2$, there are constants $c_p,c'_p>0$ such that, with probability at least $c'_p$, the Bernoulli construction produces a Fourier set $S$ of cardinality comparable to $M^{2/p}$ for which every signal supported on at most $c_pM$ points is uniquely determined, among all signals with the same sparsity bound, by its Fourier coefficients on $S^c$.

For every $\alpha>0$, there are constants $C_\alpha,c_\alpha>0$ such that, with probability at least $1/4$, Burstein's set $S$ satisfies \eqref{eq:Burstein-size} and every signal supported on at most
\begin{equation}\label{eq:Burstein-recovery-sparsity}
 \frac{c_\alpha M}{(\log M)^{C_\alpha}}
\end{equation}
points is uniquely determined, among signals satisfying the same sparsity bound, by its Fourier coefficients on $S^c$.
\end{corollary}

\begin{proof}
For the first assertion, use $\Phi(t)=t^p$ in Theorem~\ref{thm:Ryou-specialized}.  The $\Lambda_p$ constant is independent of $M$, and \eqref{eq:random-uniqueness-power} permits a fixed positive proportion of the ambient group.

For the second assertion, let $K=K_\alpha\log^{\alpha/2}(\log M)$ be the constant in \eqref{eq:Burstein-lambda}.  For the endpoint Young function,
\[
 F_{\Phi_\alpha}(\delta)
 =\delta[\Phi_\alpha^{-1}(\delta^{-1})]^2
 \simeq_\alpha [\log(e/\delta)]^{-\alpha}.
\]
If $s$ satisfies \eqref{eq:Burstein-recovery-sparsity} with $C_\alpha$ sufficiently large and $c_\alpha$ sufficiently small, then
\[
 F_{\Phi_\alpha}(2s/M)<K^{-2}.
\]
Theorem~\ref{thm:random-uniqueness} now gives the claimed sparse uniqueness on the same probability-$1/4$ event as Burstein's uniform $\Lambda_{\Phi_\alpha}$ inequality.
\end{proof}

\begin{remark}
Combining Theorem~\ref{thm:random-uniqueness} with Theorem~\ref{thm:Ryou-specialized} gives positive-probability sparse uniqueness for Fourier sets of cardinality comparable to $[\Phi^{-1}(M)]^2$.  Combining it with Theorem~\ref{thm:Burstein} gives an endpoint statement for much denser random Fourier sets.  The probability refers to the single event on which the selected set satisfies the uniform $\Lambda_\Phi$ inequality; once that event occurs, the conclusion holds simultaneously for every signal in the stated sparsity class.
\end{remark}

\section{Concluding remarks}\label{sec:conclusion}

The paper develops two complementary routes from Orlicz norm information to finite uncertainty principles.  Restriction estimates control the Fourier transform on a prescribed set and lead both to sharp support inequalities and to stable annihilating-pair estimates.  Uniform $\Lambda_\Phi$ inequalities instead control an entire Fourier-supported subspace and lead to the support-transfer formula \eqref{eq:lambda-support-main}.  The latter formula is particularly effective near $L^2$, where the precise loss in the $\Lambda_\Phi$ constant determines whether the spatial support is a positive proportion of the group or only a polylogarithmic fraction.

The probabilistic inputs are deliberately separated from the deterministic arguments.  Once a fixed set $S$ satisfies a uniform $\Lambda_\Phi$ estimate, all uncertainty and recovery conclusions hold simultaneously for every function with Fourier support in $S$.  The probability enters only in selecting such a set.  This uniformity is essential both in Theorem~\ref{thm:lambda-support} and in the random sparse-recovery statements.

Several natural questions remain.  It would be useful to determine optimal $\Lambda_\Phi$ constants for endpoint logarithmic Orlicz spaces, to understand whether Burstein's $\log^{\alpha/2}\!\log M$ loss can be improved by a non-Bernoulli construction, and to identify geometric classes of Fourier sets for which the abstract $(\Phi,\Psi)$ restriction hypotheses can be verified with constants independent of the ambient group.  These questions lie beyond the scope of the present paper.

\end{document}